\documentclass[12pt,draftcls,onecolumn]{IEEEtran}
\usepackage[ruled,boxed]{algorithm2e}

\usepackage{amsmath,graphicx,color,psfrag,mathrsfs,bm}
\usepackage{hyperref}
\usepackage{pstool}
\usepackage{longtable}          
\usepackage{amsthm}
\usepackage{amssymb}
\usepackage{tikz}
\usepackage{multicol}
\newtheorem{theorem}{Theorem}[section]
\newtheorem{corollary}[theorem]{Corollary}
\newtheorem{lemma}[theorem]{Lemma}

\theoremstyle{definition}
\newtheorem{definition}{Definition}[section]
\theoremstyle{remark}
\newtheorem{remark}{Remark}[section]

\begin{document}

\title{A Structural Separation Between Chernoff and Convex-Order Optimality in Robust Testing}

\author{%
  \IEEEauthorblockN{G\"okhan  G\"ul}\\
  \IEEEauthorblockA{Preventive Cardiology and Preventive Medicine, Department of Cardiology, University Medical Center of the Johannes Gutenberg University Mainz\\
                    Clinical Epidemiology and Systems Medicine, Center for Thrombosis and Hemostasis, University Medical Center  Johannes Gutenberg University Mainz\\
                    German Center for Cardiovascular Research (DZHK), Partner Site Rhine Main, University Medical Center of the Johannes Gutenberg University Mainz\\
                    Langenbeckstra\ss e 1, 55131 Mainz, Germany\\
                    Email: goekhan.guel@unimedizin-mainz.de}
}

\maketitle

\begin{abstract}
In classical robust hypothesis testing, least favorable distributions often
simultaneously maximize all Chernoff $u$-affinities and minimize all
$f$-divergences. This paper identifies the structural mechanism that causes
this equivalence to fail in general: the cone generated by fractional power
functions $\{x^u\}_{u\in(0,1)}$ is strictly smaller than the cone of convex
functions, inducing a separation between fractional-moment dominance and
convex-order dominance. An explicit minimal counterexample is constructed on
a three-point probability space, with convex, compact uncertainty classes and
uniformly bounded likelihood ratios, for which a single pair maximizes all
Chernoff functionals uniformly yet fails to minimize a convex
$f$-divergence. It is further proved that no such separation can occur on a
two-point space. Sufficient conditions for equivalence---including stochastic
ordering of likelihood ratios---are discussed, and an open characterization
problem in the geometry of moment cones is highlighted.
\end{abstract}

\begin{IEEEkeywords}
Chernoff functional, $f$-divergence, convex order, minimax robustness, least favorable distributions, hypothesis testing.
\end{IEEEkeywords}

\IEEEpeerreviewmaketitle

\section{Introduction}\label{sec:intro}

In robust hypothesis testing, the data-generating distributions are constrained
to lie within uncertainty classes $\mathcal{G}_0$ and $\mathcal{G}_1$, and one
seeks least favorable distributions (LFDs) that optimize performance under
worst-case conditions. Two fundamental optimality criteria govern this theory.
The Chernoff functional \cite{chernoff1952}
\[
D_u(G_0, G_1) = \int g_1^u g_0^{1-u} \, d\mu, \qquad u \in (0,1),
\]
determines exponential error decay rates and characterizes asymptotic
performance. By contrast, $f$-divergences \cite{csiszar1967}
\[
D_f(G_0, G_1) = \int f\!\left(\frac{g_1}{g_0}\right) g_0 \, d\mu,
\]
for convex $f$ with $f(1)=0$, characterize finite-sample minimax optimality
through their equivalence with stochastic ordering of likelihood ratios
\cite{huber1973, oster1978}.

For classical uncertainty classes---$\varepsilon$-contamination neighborhoods,
total variation balls, and band models---the same pair of LFDs simultaneously
maximizes all Chernoff functionals and minimizes all $f$-divergences
\cite{huber1973, huber1965, fauss2018}. This coincidence underlies the
existence of finite-sample minimax robust (FMR) tests and justifies the use
of asymptotic solutions as proxies for finite-sample optimal procedures.
%

Recent extensions to Kullback--Leibler and $\alpha$-divergence neighborhoods
\cite{levy09,gul6,gul7}, distributionally robust optimization under
$f$-divergence ambiguity sets \cite{bental2009,lam2019,duchi2021}, and
kernel-based uncertainty sets \cite{SunZou2022,SchrabKim2024} have enriched
the classical framework, but the relationship between asymptotic and finite-sample
optimality criteria remains poorly understood outside the classical settings.

The known relationship is asymmetric \cite{gul2026}: if an FMR test exists,
then an AMR test exists with the same LFDs, and the AMR solution identifies
the FMR solution uniquely when the latter exists. However, neither result
asserts that AMR guarantees FMR existence. The central question left open is
whether uniform Chernoff optimality alone---without additional geometric
constraints---suffices for finite-sample minimax robustness.

The present paper shows that the answer is negative, and identifies the
structural mechanism: the cone generated by fractional power functions
$\{x^u\}_{u \in (0,1)}$ is strictly smaller than the cone of convex functions
on any compact interval containing $1$. This strict inclusion induces a gap
between fractional-moment dominance (asymptotic optimality) and convex-order
dominance (finite-sample optimality). An explicit counterexample is constructed
on a three-point space with convex, compact, disjoint uncertainty classes and
uniformly bounded likelihood ratios: a single pair of LFDs maximizes all
Chernoff functionals uniformly, yet fails to minimize a convex $f$-divergence.
Consequently, no FMR test exists for these classes despite the presence of a
strong asymptotic solution.

The construction is minimal: no such separation can occur on a two-point
probability space, and the phenomenon requires at least three points and a
non-trivial convex uncertainty class. When the set of likelihood ratios is
totally ordered in the likelihood ratio order, equivalence is restored; this
explains why classical uncertainty classes exhibit coincidence while the
constructed example does not.

The results caution against using asymptotic solutions as automatic proxies
for finite-sample optimality. A simple diagnostic emerges: if admissible
likelihood ratios cross, equivalence may fail; if they are totally ordered,
the classical theory applies. Characterizing precisely when
fractional-moment dominance implies convex-order dominance for convex sets
of likelihood ratios remains an open problem in the geometry of moment cones.

\section{Preliminaries}\label{sec:prelim}

Let $(\Omega,\mathcal{A},\mu)$ be a measurable space. All probability measures
are absolutely continuous with respect to $\mu$, with densities denoted by
lower-case letters. For densities $g_0,g_1$, define the likelihood ratio
$L=g_1/g_0$ (with the convention $0/0=0$).

\subsection{Chernoff functionals and asymptotic minimax robustness}

The Chernoff functional of order $u\in(0,1)$ is defined by
\begin{equation}\label{eq:chernoff}
D_u(G_0,G_1)=\int_{\Omega}\Bigl(\frac{g_1}{g_0}\Bigr)^{u}g_0\,d\mu
=\int_{\Omega}g_1^{u}g_0^{1-u}\,d\mu.
\end{equation}

\begin{definition}[Asymptotic minimax robustness]
A test is \emph{asymptotically minimax robust} (AMR) if there exist LFDs
$(\hat{G}_0,\hat{G}_1)$ and a parameter $u^{*}\in(0,1)$ such that
$(\hat{G}_0,\hat{G}_1)$ maximizes $D_{u^{*}}$ over
$\mathcal{G}_0\times\mathcal{G}_1$, and the error exponents under
$(\hat{G}_0,\hat{G}_1)$ dominate those under any other pair.
\end{definition}

Uniform asymptotic minimax robustness requires the same pair to maximize
$D_u$ for \emph{every} $u\in(0,1)$.

\subsection{$f$-divergences and finite-sample minimax robustness}

For a convex function $f:(0,\infty)\to\mathbb{R}$ with $f(1)=0$, the
$f$-divergence is defined by
\begin{equation}\label{eq:fdiv}
D_f(G_0,G_1)=\int_{\Omega}f\Bigl(\frac{g_1}{g_0}\Bigr)g_0\,d\mu.
\end{equation}

\begin{definition}[Single-sample minimax robustness]
A pair $(\hat{G}_0,\hat{G}_1)$ is \emph{single-sample minimax robust}
(SMR) if the likelihood ratio $\hat{l}=\hat{g}_1/\hat{g}_0$ satisfies
\begin{equation}\label{eq:smr}
G_0[\hat{l}<t]\ge\hat{G}_0[\hat{l}<t],\qquad
G_1[\hat{l}<t]\le\hat{G}_1[\hat{l}<t],
\end{equation}
for all $t\in\mathbb{R}$ and all $(G_0,G_1)\in\mathcal{G}_0\times
\mathcal{G}_1$.
\end{definition}

The SMR inequalities \eqref{eq:smr} are equivalent to minimization of all
$f$-divergences:

\begin{theorem}[\cite{huber1973, oster1978}]\label{thm:huber}
A pair $(\hat{G}_0,\hat{G}_1)$ satisfies the SMR inequalities
\eqref{eq:smr} if and only if it minimizes $D_f(G_0,G_1)$ over
$\mathcal{G}_0\times\mathcal{G}_1$ for every twice continuously
differentiable convex $f$ with $f(1)=0$.
\end{theorem}

A test is \emph{finite-sample minimax robust} (FMR) if it is SMR for every
sample size $n<\infty$.

\subsection{The asymptotic--finite-sample link}

The following results from \cite{gul2026} summarize the known relationship
between AMR and FMR:

\begin{theorem}[\cite{gul2026}]\label{thm:fmr-implies-amr}
If an FMR test exists, then an AMR test exists for the same uncertainty
classes and with the same LFDs.
\end{theorem}

\begin{theorem}[\cite{gul2026}]\label{thm:amr-implies-fmr}
Let $(\hat{G}_0,\hat{G}_1)$ be AMR LFDs corresponding to the minimizing
parameter $u^{*}$. Assume that the maximizer of $D_{u^{*}}$ is unique. If an
FMR test exists for the same uncertainty classes, then its LFDs coincide
$\mu$-a.e.\ with $(\hat{G}_0,\hat{G}_1)$.
\end{theorem}

Theorem~\ref{thm:fmr-implies-amr} shows that FMR implies AMR.
Theorem~\ref{thm:amr-implies-fmr} shows that AMR \emph{identifies} the FMR
solution when FMR exists. Neither theorem asserts that AMR \emph{guarantees}
FMR existence. The question addressed in this paper is:

\begin{quote}
\emph{Does the existence of an AMR solution---even one that maximizes all
$D_u$ uniformly---imply the existence of an FMR solution?}
\end{quote}

\section{The structural gap}\label{sec:gap}

Before presenting the counterexample, the geometric obstruction responsible
for the separation is clarified.

\subsection{Fractional moments and convex order}

Let $\mathcal{L}$ be a set of nonnegative random variables with mean $1$,
representing achievable likelihood ratios. For $L\in\mathcal{L}$, define the
fractional moments $m_u(L)=\mathbb{E}[L^u]$ for $u\in(0,1)$.

\begin{definition}
A random variable $L$ \emph{dominates} $L'$ in the \emph{fractional-moment
order}, written $L\succeq_{\mathrm{fm}}L'$, if $m_u(L)\ge m_u(L')$ for all
$u\in(0,1)$.
\end{definition}

\begin{definition}
A random variable $L$ \emph{dominates} $L'$ in the \emph{convex order},
written $L\succeq_{\mathrm{cx}}L'$, if $\mathbb{E}[f(L)]\ge\mathbb{E}[f(L')]$
for all convex $f$ with $f(1)=0$.
\end{definition}

By Theorem~\ref{thm:huber}, minimizing all $f$-divergences is equivalent to
being minimal in the convex order. Maximizing all Chernoff functionals is
equivalent to being maximal in the fractional-moment order. Since $x\mapsto
x^u$ is concave for $u\in(0,1)$, maximizing $\mathbb{E}[L^u]$ corresponds to
extremality in the concave order, which is the dual of the convex order.
Minimality in the convex order is therefore a strictly stronger condition
than maximality in the fractional-moment order.

The key observation is that the cone of functions
\begin{equation}\label{eq:cone-power}
\mathcal{C}_{\mathrm{pow}}=\operatorname{cone}\{x^u:u\in(0,1)\}
\end{equation}
is strictly smaller than the cone of concave functions on any interval
$[a,b]$ with $a<1<b$. Consequently, dominance in all fractional moments does
not imply dominance in the concave (or convex) order.

\begin{remark}
The strict inclusion $\mathcal{C}_{\mathrm{pow}}\subsetneq\mathcal{C}_{\mathrm{conc}}$ 
is classical; power functions form a one-parameter subfamily of the 
infinitely generated concave cone.
\end{remark}

\subsection{Implications for robust testing}

The gap between $\mathcal{C}_{\mathrm{pow}}$ and $\mathcal{C}_{\mathrm{conc}}$
translates into the following: control of all fractional moments of the
likelihood ratio does not guarantee control of all convex functionals.
Therefore, a pair that is optimal for asymptotic error exponents (Chernoff)
need not be optimal for finite-sample worst-case risk ($f$-divergence).

The classical positive examples ($\varepsilon$-contamination, total variation,
band models) avoid this gap because their likelihood ratio sets are totally
ordered in the likelihood ratio order, which forces the fractional-moment
maximizer to coincide with the convex-order minimizer. The counterexample in
the next section exploits the first geometric configuration where this ordering
fails; see Figure~\ref{fig:cone-gap}.

\begin{figure}[t]
\centering
\begin{tikzpicture}[scale=1]
    \draw[thick, fill=gray!15] (0,0) -- (5,3) -- (5,-3) -- cycle;
    \node at (4,0) {\large $\mathcal{C}_{\mathrm{conc}}$};
    
    \draw[thick, fill=gray!50] (0,0) -- (2.2,1) -- (2.2,-1) -- cycle;
    \node at (1.3,0) {$\mathcal{C}_{\mathrm{pow}}$};
    
    \node[right] at (5,2.5) {\small concave functions};
    \node[right] at (2.2,0.8) {\small $x^u$};
    
    \draw[->, thick, >=stealth] (2.4,0.3) -- (3.6,0.3);
    \node at (3,0.6) {\small gap};
\end{tikzpicture}
\caption{Schematic illustration of the strict inclusion 
$\mathcal{C}_{\mathrm{pow}}\subsetneq\mathcal{C}_{\mathrm{conc}}$, where 
$\mathcal{C}_{\mathrm{pow}}$ is the cone generated by fractional power 
functions and $\mathcal{C}_{\mathrm{conc}}$ is the cone of concave functions. 
Dominance in all power functions (fractional-moment order) does not imply 
dominance in all concave functions; equivalently, by duality, it does not 
imply convex-order dominance.}
\label{fig:cone-gap}
\end{figure}
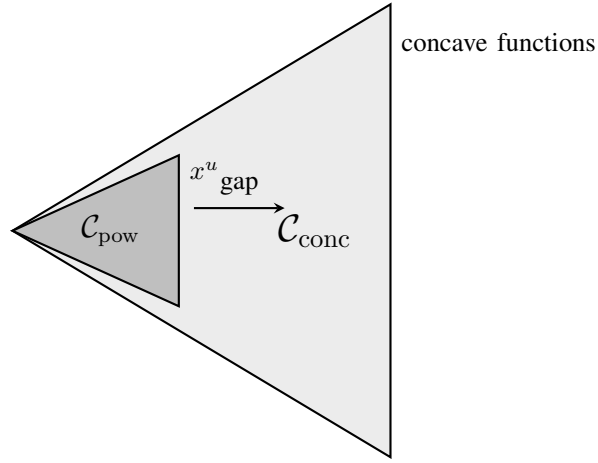

\section{The counterexample}\label{sec:counterexample}

An explicit counterexample is now constructed on a three-point space with
counting measure, a singleton $\mathcal{G}_0$, and $\mathcal{G}_1$ equal to
the convex hull of two probability vectors. The construction enforces
second-order contact at $u=1$, ensuring uniform Chernoff optimality while
allowing divergence at higher-order convex functionals.

\subsection{Construction}

Let $\Omega=\{1,2,3\}$ with counting measure $\mu$. Define the probability
vectors
\begin{align}
g_0      &= (0.40194407,\; 0.47871677,\; 0.11933916),\label{eq:g0}\\
\hat{g}_1 &= (0.07550977,\; 0.67274530,\; 0.25174493),\label{eq:ghat}\\
g_1'     &= (0.07295949,\; 0.68792995,\; 0.23911056).\label{eq:g1prime}
\end{align}

Set $\mathcal{G}_0=\{g_0\}$ and $\mathcal{G}_1=\operatorname{conv}
\{\hat{g}_1,g_1'\}$. Both sets are convex and compact. They are disjoint
because $g_0$ does not lie on the line segment between $\hat{g}_1$ and
$g_1'$.

The likelihood ratios are $L_i=\hat{g}_{1,i}/g_{0,i}$ and
$L_i'=g_{1,i}'/g_{0,i}$:
\begin{align}
L  &= (0.18786138,\; 1.40530966,\; 2.10949139),\label{eq:L}\\
L' &= (0.18151652,\; 1.43702916,\; 2.00362195).\label{eq:Lprime}
\end{align}

They satisfy $0.18<L,L'<2.11$ and $\mathbb{E}_{g_0}[L]=
\mathbb{E}_{g_0}[L']=1$, so the likelihood ratios are uniformly bounded and
properly normalized. The geometry is illustrated in Figure~\ref{fig:simplex}.

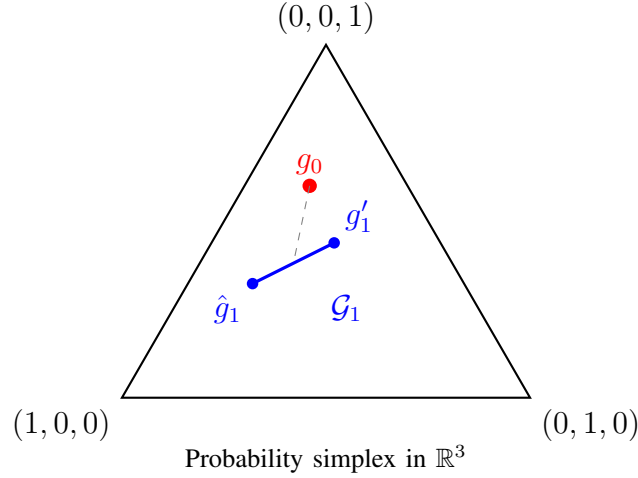
\begin{figure}[t]
\centering
\begin{tikzpicture}[scale=5.4]
    \coordinate (A) at (0,0);
    \coordinate (B) at (1,0);
    \coordinate (C) at (0.5,0.866);
    \draw[thick] (A) -- (B) -- (C) -- cycle;
    \node[below left] at (A) {$(1,0,0)$};
    \node[below right] at (B) {$(0,1,0)$};
    \node[above] at (C) {$(0,0,1)$};
    \coordinate (ghat) at (0.32,0.28);
    \coordinate (gprime) at (0.52,0.38);
    \draw[very thick, blue] (ghat) -- (gprime);
    \fill[blue] (ghat) circle (0.4pt) node[below left] {$\hat{g}_1$};
    \fill[blue] (gprime) circle (0.4pt) node[above right] {$g_1'$};
    \node[blue] at (0.55,0.22) {$\mathcal{G}_1$};
    \coordinate (g0) at (0.46,0.52);
    \fill[red] (g0) circle (0.5pt) node[above] {$g_0$};
    \draw[dashed, gray, thin] (g0) -- (0.42,0.33);
    \node at (0.5,-0.15) {\small Probability simplex in $\mathbb{R}^3$};
\end{tikzpicture}
\caption{The uncertainty class $\mathcal{G}_1 = \operatorname{conv}\{\hat{g}_1, g_1'\}$ 
(blue segment) and the fixed null distribution $g_0$ (red). The disjointness 
condition $g_0 \notin \mathcal{G}_1$ is geometrically apparent.}
\label{fig:simplex}
\end{figure}

\begin{remark}
The construction enforces second-order contact of $\phi(u)$ at $u=1$:
\[
\sum_{i}v_{i}=0,\qquad\sum_{i}(\log L_{i})v_{i}=0,\qquad
\sum_{i}(\log L_{i})^{2}v_{i}<0,
\]
for perturbations $v=g_1'-\hat{g}_1$. These conditions ensure
$\phi(1)=\phi'(1)=0$ and $\phi''(1)<0$. Candidate points were generated along
such directions while preserving positivity; Chernoff dominance was verified
analytically via Lemma~\ref{lem:sign}, and $f$-divergence violation was
detected using hinge functions $f(x)=(x-t)_+$.
\end{remark}

\subsection{Uniform Chernoff dominance}

For $u\in(0,1)$, the directional derivative of $D_u(g_0,\cdot)$ at
$\hat{g}_1$ toward $g_1'$ equals $u\cdot\phi(u)$, where
\begin{equation}\label{eq:phi}
\phi(u)=\sum_{i=1}^{3}L_i^{u-1}(g_{1,i}'-\hat{g}_{1,i}).
\end{equation}

Since $g_1\mapsto D_u(g_0,g_1)$ is concave, verifying $\phi(u)<0$ for all
$u\in(0,1)$ ensures that $\hat{g}_1$ is the unique maximizer over
$\mathcal{G}_1$ for every $u$.

\begin{lemma}\label{lem:sign}
Let $\psi(t)=c_1 e^{\lambda_1 t}+c_2 e^{\lambda_2 t}+c_3 e^{\lambda_3 t}$
with $\lambda_1<\lambda_2<\lambda_3$, and define $g(t)=\psi(t)
e^{-\lambda_1 t}$. If $\psi(0)=\psi'(0)=0$ and $\psi''(0)<0$, then
$\psi(t)<0$ for all $t<0$.
\end{lemma}

\begin{proof}
We have $g(t)=c_1+c_2 e^{\mu_2 t}+c_3 e^{\mu_3 t}$ where
$\mu_j=\lambda_j-\lambda_1>0$. Since $e^{-\lambda_1 t}>0$, it suffices to
show $g(t)<0$ for all $t<0$.

From the assumptions, $g(0)=g'(0)=0$ and $g''(0)<0$. The derivative
$g'(t)=c_2\mu_2 e^{\mu_2 t}+c_3\mu_3 e^{\mu_3 t}$ vanishes at $t=0$ and,
being a sum of two exponentials with distinct rates, has at most one zero.
Hence $g'(t)\neq 0$ for all $t\neq 0$. As $g''(0)<0$, we have $g'(h)>0$ for
small $h<0$, so $g'(t)>0$ for all $t\in(-\infty,0)$. Therefore $g$ is
strictly increasing on $(-\infty,0]$, and $g(t)<g(0)=0$ for all $t<0$.
\end{proof}

Set $\lambda_i=\log L_i$, $c_i=g_{1,i}'-\hat{g}_{1,i}$, and $t=u-1$. Then
$\phi(u)=\psi(t)$ with $\psi$ as in Lemma~\ref{lem:sign}. A direct
calculation yields:
\begin{itemize}
\item $\phi(1)=\sum_{i=1}^{3}(g_{1,i}'-\hat{g}_{1,i})=0$;
\item $\phi'(1)=\sum_{i=1}^{3}(\log L_i)(g_{1,i}'-\hat{g}_{1,i})=0$;
\item $\phi''(1)=\sum_{i=1}^{3}(\log L_i)^{2}(g_{1,i}'-\hat{g}_{1,i})
=-0.01241<0$.
\end{itemize}

Applying Lemma~\ref{lem:sign} with $t=u-1$ yields $\phi(u)<0$ for all
$u\in(0,1)$. Since the directional derivative of the concave function
$D_u(g_0,\cdot)$ at $\hat{g}_1$ along $\mathcal{G}_1$ is strictly negative,
$\hat{g}_1$ is the unique maximizer over $\mathcal{G}_1$ for every
$u\in(0,1)$. Therefore $(g_0,\hat{g}_1)$ maximizes all Chernoff functionals
over $\mathcal{G}_0\times\mathcal{G}_1$.

\subsection{$f$-divergence violation}

Take the convex function $f(x)=(x-1.43448)_+$, which satisfies $f(1)=0$.
Then
\begin{align}
D_f(g_0,\hat{g}_1)-D_f(g_0,g_1')
&= \sum_{i=1}^{3}g_{0,i}\bigl[f(L_i)-f(L_i')\bigr]\nonumber\\
&= 0.11933916\times(2.10949139-1.43448)\nonumber\\
&\quad-\bigl[0.47871677\times(1.43702916-1.43448)\nonumber\\
&\qquad\quad+0.11933916\times(2.00362195-1.43448)\bigr]\nonumber\\
&\approx 0.0114>0.\label{eq:fviolation}
\end{align}

Thus $(g_0,\hat{g}_1)$ does \emph{not} minimize this $f$-divergence over
$\mathcal{G}_0\times\mathcal{G}_1$.

\subsection{Main theorem}

\begin{theorem}\label{thm:main}
There exist convex, compact, disjoint sets $\mathcal{G}_0,
\mathcal{G}_1$ of probability measures with uniformly bounded likelihood
ratios $0<a\le L\le b<\infty$ such that:
\begin{enumerate}
\item[\rm(i)] the pair $(g_0,\hat{g}_1)\in\mathcal{G}_0\times
\mathcal{G}_1$ maximizes $D_u(G_0,G_1)$ for every $u\in(0,1)$;
\item[\rm(ii)] the pair $(g_0,\hat{g}_1)$ does not minimize
$D_f(G_0,G_1)$ for the convex function $f(x)=(x-1.43448)_+$;
\item[\rm(iii)] consequently, no finite-sample minimax robust test exists for
$\mathcal{G}_0,\mathcal{G}_1$.
\end{enumerate}
\end{theorem}

\begin{proof}
Parts (i) and (ii) follow from Sections \ref{sec:counterexample}.2 and
\ref{sec:counterexample}.3. For (iii), Theorem~\ref{thm:huber} implies that
any FMR pair minimizes all $f$-divergences. By (i), $(g_0,\hat{g}_1)$ is the
unique maximizer of all $D_u$, so any FMR pair must coincide with it. By
(ii), this pair fails to minimize the hinge $f$-divergence, a contradiction.
\end{proof}

\section{Minimality of the construction}\label{sec:minimality}

The counterexample is minimal in two senses: the number of points in the
sample space, and the complexity of $\mathcal{G}_1$.

\subsection{Two points are insufficient}

\begin{theorem}\label{thm:2point}
Let $\Omega=\{1,2\}$ with counting measure. Let $\mathcal{G}_0=\{g_0\}$
and $\mathcal{G}_1=\operatorname{conv}\{g_1,g_1'\}$ be disjoint convex
sets of probability vectors with $g_0\notin\mathcal{G}_1$. If
$\hat{g}_1\in\mathcal{G}_1$ maximizes $D_u(g_0,g_1)$ for all
$u\in(0,1)$, then $(g_0,\hat{g}_1)$ minimizes every convex
$f$-divergence over $\mathcal{G}_0\times\mathcal{G}_1$.
\end{theorem}

\begin{proof}
Write $g_0=(p,1-p)$ with $p\in(0,1)$. Any $g_1\in\mathcal{G}_1$ induces a
likelihood ratio $L=(a,b)$ with $pa+(1-p)b=1$. As $g_1$ varies over the line
segment $\mathcal{G}_1$, the first coordinate $a$ varies over a closed
interval $[a_{\min},a_{\max}]$ not containing $1$. The second coordinate is
determined by $b=(1-pa)/(1-p)$.

For $u\in(0,1)$, the Chernoff functional is
\[
D_u(a)=pa^u+(1-p)\Bigl(\frac{1-pa}{1-p}\Bigr)^u,
\]
a strictly concave function of $a$. Because $a=1$ is excluded, the maximum
on $[a_{\min},a_{\max}]$ is attained at the endpoint closest to $1$. This
endpoint is independent of $u$, so the same $\hat{g}_1$ maximizes $D_u$ for
every $u\in(0,1)$.

The variance of the likelihood ratio is
\[
\operatorname{Var}(L)=p(a-1)^2+(1-p)(b-1)^2
=\frac{p}{1-p}(a-1)^2,
\]
a strictly increasing function of $|a-1|$. Since the maximizer uses the
endpoint closest to $1$, the corresponding likelihood ratio has the smallest
variance among all achievable likelihood ratios. On a two-point space with
fixed mean, variance ordering coincides with the convex order. Hence
$\hat{L}$ is minimal in the convex order, and by Theorem~\ref{thm:huber},
$(g_0,\hat{g}_1)$ minimizes every convex $f$-divergence.
\end{proof}

\begin{corollary}
No counterexample to the implication ``uniform Chernoff maximization $\Rightarrow$
$f$-divergence minimization'' can be constructed on a two-point probability
space.
\end{corollary}

Equivalently, on a two-point space the fractional-moment order and the convex 
order coincide.

\subsection{Singleton uncertainty classes are trivial}

If $\mathcal{G}_1=\{g_1\}$ is a singleton, then $(g_0,g_1)$ trivially
minimizes all $f$-divergences. Thus a non-trivial counterexample requires a non-singleton convex uncertainty class. The convex hull of two points is the smallest non-trivial convex set.

\begin{remark}
Theorem~\ref{thm:2point} shows that the obstruction requires at least three
points \emph{and} at least two generating measures for $\mathcal{G}_1$. The
counterexample in Section~\ref{sec:counterexample} uses exactly three points
and $\mathcal{G}_1=\operatorname{conv}\{\hat{g}_1,g_1'\}$, so it is minimal
in both the sample-space dimension and the uncertainty-class complexity.
\end{remark}

\section{Sufficient conditions for equivalence}\label{sec:sufficient}

The classical positive examples share a common property: the least favorable
likelihood ratio is stochastically ordered relative to all other admissible
pairs. Specifically, $\hat{L}$ satisfies the SMR inequalities \eqref{eq:smr},
which is equivalent to minimizing all $f$-divergences by
Theorem~\ref{thm:huber}.

In classical uncertainty classes---$\varepsilon$-contamination neighborhoods,
total variation balls, and band models---the least favorable likelihood ratio
is an envelope function that bounds all admissible likelihood ratios. This
envelope structure ensures simultaneous maximization of all Chernoff
functionals and minimization of all $f$-divergences. The counterexample in
Section~\ref{sec:counterexample} shows that when no such envelope structure
exists, the equivalence breaks.

\section{Discussion}\label{sec:discussion}

\subsection{Implications for robust test design}

The results show that asymptotic solutions cannot be used as automatic
proxies for finite-sample optimality, even when they are uniform across all
Chernoff orders. A simple diagnostic emerges: if the admissible likelihood
ratios cross, the equivalence may fail; if they are totally ordered, the
classical theory applies.

Crossing likelihood ratios arise naturally in modern robust inference whenever
uncertainty classes are constructed from data-driven divergence balls or moment
constraints rather than from classical $\varepsilon$-contamination
neighborhoods. In such settings, the admissible likelihood ratios need not form
a monotone family, and the diagnostic developed here becomes essential.

This diagnostic is essentially minimal: The counterexample shows that crossing occurs
already with three points and two measures---the minimal configuration
admitting a separation. Therefore, practitioners designing robust tests for
non-standard uncertainty models should verify whether the likelihood ratio set
is totally ordered before assuming that an asymptotic solution carries
finite-sample guarantees.

\subsection{The moment cone characterization problem}

The core open problem raised by this work is the characterization of when
fractional-moment dominance implies convex-order dominance. Formally:

\begin{quote}
\emph{Let $\mathcal{L}$ be a convex set of nonnegative random variables with
mean $1$. Under what conditions on $\mathcal{L}$ does $L\succeq_{\mathrm{fm}}
L'$ for all $L'\in\mathcal{L}$ imply $L\succeq_{\mathrm{cx}}L'$ for all
$L'\in\mathcal{L}$?}
\end{quote}

The results show:
\begin{itemize}
\item Sufficient: $\mathcal{L}$ totally ordered in the likelihood ratio order.
\item Crossing likelihood ratios suffice to produce separation
(Theorem~\ref{thm:main}).
\end{itemize}

The gap between these conditions is narrow but non-empty. A complete
characterization would likely involve the Choquet theory of the cone
$\mathcal{C}_{\mathrm{pow}}$ and its extreme rays.

\subsection{Extension to infinite-dimensional classes}

The counterexample is finite-dimensional, but the phenomenon persists for
infinite-dimensional classes. Kullback--Leibler divergence balls provide a
known instance where the Chernoff maximizer depends on $u$, so the uniform
condition fails. It would be interesting to construct an infinite-dimensional
counterexample where a uniform Chernoff maximizer exists but fails to minimize
an $f$-divergence. Such a construction might involve densities with
interleaved crossing regions analogous to the three-point pattern exhibited
here.

\section{Conclusion}

Uniform maximization of all Chernoff functionals does not imply minimization
of all $f$-divergences, even under convexity, compactness, disjointness, and
bounded likelihood ratios. This exhibits a structural separation between
fractional-moment dominance and convex-order dominance, and provides the first
explicit example where asymptotic minimax robustness exists in its strongest
form while finite-sample minimax robustness does not.

The counterexample is minimal: it uses three points and the convex hull of
two measures, and no counterexample can exist on a two-point space. The
obstruction is the strict inclusion of the power-function cone in the
convex-function cone, which becomes manifest as soon as the likelihood ratio
set admits crossing elements.

The results caution against using asymptotic solutions as automatic proxies
for finite-sample optimality, and identify the stochastic ordering of
likelihood ratios as the key structural property governing the relationship
between the two regimes. Characterizing the exact boundary between equivalence
and separation remains an open problem in the geometry of moment cones. 
The classical equivalence between Chernoff and $f$-divergence optimality in
robust testing is therefore not a consequence of asymptotic minimaxity itself,
but rather of additional geometric ordering properties of the admissible
likelihood ratio family. When these ordering properties fail, the equivalence
fails with them.

The present results suggest that quantization may fundamentally alter the relationship between asymptotic and finite-sample robustness. In decentralized detection, binary quantization collapses the likelihood-ratio geometry to a one-dimensional structure on which fractional-moment and convex-order optimality coincide, whereas larger alphabets permit separation. Characterizing quantizers that preserve or restore this equivalence remains an open problem.

\end{document}